\documentclass[10pt, leqno]{amsart}

\usepackage{Style}

\title{On monomial resonance}

\author{C\u alin Spiridon}

\address{University of Bucharest, Faculty of Mathematics and Informatics, Academiei Str. 14 Bucharest, Romania \& "Simion Stoilow" Institute of Mathematics of the Romanian Academy, P.O. Box 1-764, Bucharest, Romania}

\email{cspiridon@imar.ro}

\thanks{The author have been partly funded by the project PNRR-III-C9-2022-I8 "Cohomological Hall algebras of smooth surfaces and applications" - CF 44/14.11.2022.}
\thanks{The author would like to thank professor Marian Aprodu for suggesting the problem.}

\date{\today}

\begin{document}

\begin{abstract}
This note addresses the resonance of monomial subspaces. In the first part, we completely characterize the cases where the Fitting scheme structure of the resonance associated with a monomial subspace is reduced, and we further investigate the primary decomposition of the corresponding Fitting ideal. The second part focuses on non-monomial subspaces having the resonance of a monomial subspace.
\end{abstract} 

\maketitle

%\tableofcontents

\section{Introduction}

Let $V$ be a finite dimensional vector space and $K \subseteq \bigwedge^2V$ a linear subspace. In \cite{papadima_vanishing_2015}, Papadima and Suciu attached to any such pair $(V,K)$ two closely related objects: an algebraic object, $W(V,K)$, which is a finitely generated graded module over the symmetric algebra of $V$, and a geometric one, $\cR(V,K)$, which is the support locus of $W(V,K)$. They are called the \emph{Koszul module} and the \emph{resonance variety} of the pair $(V,K)$ and both emerged from geometric group theory. The Koszul module generalize the infinitesimal Alexander invariant \cite{papadima_2004} and the resonance variety extend the notion of first resonance variety of a group. A sharp vanishing theorem for the graded pieces of the Koszul modules led to a new proof of generic Green conjecture \cite{aprodu_green_2019} and also to a better understanding of some important invariants of finitely generated groups, such as the Chen ranks, the degree of growth or virtual nilpotency class \cite{aprodu_topological_2022}. Other applications in algebraic geometry were provided by \cite{aprodu_vanishing_2024}. Due to the fact that the resonance $\cR(V,K)$ is the support of the Koszul module, one can endow the affine variety $\cR(V,K)$ with natural scheme structures, using either the annihilator ideal or the Fitting ideal of $W(V,K)$. The annihilator scheme structure of the resonance is considered in \cite{aprodu_reduced_2024}, since it has the advantage of being invariant under closed embeddings of the ambient affine spaces. However, the Fitting scheme structure of the resonance may capture more subtle algebraic or combinatorial information about the linear subspace $K$.

A particular and interesting case is when $V$ admits a basis $\{e_1, \ldots, e_n \}$ such that $K$ is \emph{monomial}, that is $K$ is generated by elements of the form $e_i \wedge e_j$. For this setup, a graph $\Gamma = (\mathsf{V}, \mathsf{E})$ can be naturally associated to the pair $(V,K)$, by taking $\mathsf{V} = \{1,\ldots, n\}$ and $\mathsf{E} = \{(i,j)\in \binom{\mathsf{V}}{2} : e_i \wedge e_j \in K \}$. We will denote the corresponding Koszul module and resonance variety by $W_\Gamma$ and $\cR_\Gamma$. The resonance varieties obtained starting from monomial subspaces recover the resonance of right-angled Artin groups, see \cite{papadima_2006}. It was recently shown \cite{aprodu_higher_2024} that the annihilator scheme structure of $\cR_\Gamma$ is reduced. 

\medskip

In the first part of this note, we completely characterize the cases where the Fitting scheme structure of $\cR_\Gamma$ is reduced, see Theorem \ref{thm_fitting}. We also describe the minimal components of the Fitting ideal $\fitt W_\Gamma$, when $\Gamma$ is a tree, Theorem \ref{thm_primary_decomp_trees}, and propose the Conjecture \ref{conj_minimal components} for arbitrary graphs, highlighting what additional combinatorial information does the Fitting scheme structure of the resonance provide compared with the annihilator one. If $| \overline{\mathsf{E}} | \ge n - 2$, we further propose a stronger Conjecture \ref{conj_fitt_ideal}, which determines the structure of the ideal $\fitt W_\Gamma$ in terms of the spanning trees of the graph $\Gamma$.

In the second part of the note, we address the question of whether a monomial subspace can have the same resonance as a non-monomial subspace. Some sufficient conditions on the graph $\Gamma$ under which this situation occurs are provided in Propositions \ref{prop_triangle} and \ref{prop_disconnecting_vertex}. We also show in Proposition \ref{prop_C4} that for the cycle graph on $4$ vertices $C_4$ there is no non-monomial subspace with the same resonance. 

\section{Brief review of Koszul modules and resonance varieties}
\label{sec:2}

Fix $\bk$ an algebraically closed field of characteristic $0$. Let $V$ be a finite dimensional $\bk$-vector space and $K \subseteq \bigwedge^2V$ a linear subspace in the second exterior power of $V$. Denote by $S = \sym(V)$ the symmetric algebra of $V$. The \emph{Koszul module} associated with the pair $(V,K)$ is the following graded $S$-module \cite[\S 2]{papadima_vanishing_2015}, \cite[Definition 2.1]{aprodu_topological_2022}:
\begin{equation} \label{eq:presentation}
W(V,K) = \coker \left( \displaystyle \textstyle{\bigwedge^3 V} \otimes S(-1) \xrightarrow{\ \pi\circ\delta_3\ } \displaystyle\frac{\textstyle{\bigwedge^2 V}}{K} \otimes S\right)   
\end{equation}
where $\pi: \bigwedge^2 V \otimes S \twoheadlongrightarrow (\bigwedge^2 V/K)\otimes S$ is the canonical projection and $\delta_3: \bigwedge^3 V \otimes S(-1) \longrightarrow \bigwedge^2 V \otimes S$ defined by 
\[
e_1 \wedge e_2 \wedge e_3 \otimes f \longmapsto e_1 \wedge e_2 \otimes e_3f - e_1 \wedge e_3 \otimes e_2f + e_2 \wedge e_3 \otimes e_1f
\]
is the third Koszul differential of the classical Koszul exact complex, which gives the minimal graded free resolution of $\bk \cong S/(V)$, see e.g. \cite{eisenbud_syzygies_2005}.

Now, let $(\--)^\vee: \operatorname{Vect}_\bk \longrightarrow \operatorname{Vect}_\bk$ be the functor sending a $\bk$-vector space to its dual and $K^\perp = (\bigwedge^2V/K)^\vee$ the space of skew-symmetric bilinear forms on $V$ vanishing on $K$. It was shown in \cite[Lemma 2.4]{papadima_vanishing_2015} that, if non-empty, the set-theoretic support of the Koszul module $W(V,K)$ is the \emph{resonance locus} of the pair $(V,K)$:
\begin{equation}
\cR(V,K) = \{a\in V^\vee : \exists \ b\in V^\vee \text{ such that } 0 \neq a\wedge b \in K^\perp \} \cup \{0 \}
\end{equation}
When endowed with the reduced structure, we refer to $\cR(V,K)$ as the \emph{resonance variety}. Nevertheless, the property of being the support of the Koszul module allows us to define a natural scheme structure on the resonance locus, by using the annihilator ideal of $W(V,K)$, namely, $\cR^{\ann}(V,K) = \operatorname{Spec}(S/\ann W(V,K))$, which is analyzed in \cite{aprodu_reduced_2024}. Taking into account the presentation (\ref{eq:presentation}) of the Koszul module, we can endow the resonance with another natural scheme structure, given by the ideal of maximal minors of the presentation matrix, that is the Fitting ideal $\fitt W(V,K)$, see \cite[\href{https://stacks.math.columbia.edu/tag/07Z6}{Tag 07Z6}]{stacks-project}. We obtain in this way $\cR^{\operatorname{Fitt}}(V,K) = \operatorname{Spec}(S/\fitt W(V,K))$. It is the aim of this note to get a better understanding of this latter scheme structure in a monomial setup.

\section{Monomial resonance}

Throughout this section, $V$ will be a vector space of dimension $n\ge 3$ over an algebraically closed field $\bk$ of characteristic $0$ and $K \subseteq \bigwedge^2V$ a linear subspace.

\subsection{First properties}

Describing the resonance of a pair $(V,K)$ is a difficult task in general. However, in some particular instances, the resonance variety is well understood. One of these cases is when the subspace $K$ is \emph{monomial}, that is, there exists a basis $\{e_1,\ldots, e_n\}$ of $V$ such that $K$ is generated by elements of the form $e_i \wedge e_j$. An undirected simple graph $\Gamma = (\mathsf{V}, \mathsf{E})$ can be easily attached to this setup in the following way: set $\mathsf{V} = \{1, \ldots, n\}$ and $\mathsf{E} = \{(i,j)\in \binom{\mathsf{V}}{2} : e_i \wedge e_j \in K \}$, where $\binom{\mathsf{V}}{2}$ denotes the set of ordered pairs from $\mathsf{V}$. We refer to the resonance variety of such a pair $(V,K)$ as a \emph{monomial resonance} and denote it by $\cR_\Gamma$. It turns out that $\cR_\Gamma$ is nothing but the resonance of the right-angled Artin group $G_\Gamma$, see \cite{papadima_2006}, \cite[Section 8]{aprodu_reduced_2024}, which is presented as
\[
G_\Gamma :=  \langle\ g_1, \ldots, g_n \ : \ g_i \cdot g_j = g_j \cdot g_i \ \text{for all} \ (i,j)\in \mathsf{E} \ \rangle
\]
The resonance variety of $G_\Gamma$ has been described in \cite[Theorem 5.5]{papadima_2006}, as a union of linear subspaces of $V^\vee$:
\begin{equation} \label{eq:res_coordinate}
\cR_\Gamma = \bigcup \Bigl\{V^\vee_{\Gamma'} : \Gamma' \ \text{is a maximally disconnected full subgraph of}\ \Gamma \Bigr\}   
\end{equation}
where $V^\vee_{\Gamma'}$ is the coordinate subspace of $V^\vee$ generated by the vertices of $\Gamma'$.

Now consider $S = \bk[x_1, \ldots, x_n]$ the coordinate ring of $V^\vee$. For an arbitrary subset $\mathsf{U} \subseteq \mathsf{V}$, we define $I_{\mathsf{U}}$ to be the ideal of $S$ generated by the variables $x_i$, with $i\in \mathsf{U}$. By convention, $I_\emptyset = (0)$. We denote by $I_\Gamma \subseteq S$ the ideal of the resonance variety $\cR_\Gamma \subseteq V^\vee$. 

\begin{defn} \label{defn_disconnecting}
A \emph{disconnecting set of vertices} of $\Gamma$ is a (possibly empty) subset $\mathsf{U} \subseteq \mathsf{V}$ such that the full subgraph of $\Gamma$ on the vertex set $\mathsf{V} \setminus \mathsf{U}$ is not connected. %A disconnecting set of $\Gamma$ which is minimal with respect to the inclusion is called \emph{minimal}. 
We denote by $\cD_\Gamma$ the set of \emph{minimal} (with respect to the inclusion) disconnecting sets of $\Gamma$.
\end{defn}

We may reformulate (\ref{eq:res_coordinate}) in the following way.

\begin{thm} \label{thm_structure_ideal_res}

The ideal of the resonance variety $\cR_\Gamma$ has the following primary decomposition:
\[
I_\Gamma = \bigcap_{\mathsf{U}\in \cD_\Gamma} I_{\mathsf{U}}
\]
\end{thm}

In particular,

\begin{cor} \label{cor_neconex}
The graph $\Gamma$ is disconnected if and only if $I_\Gamma = (0)$ if and only if $\cR_\Gamma = V^\vee$.
\end{cor}

Let us denote by $W_\Gamma$ the corresponding Koszul module. In \cite[Theorem 5.2]{aprodu_higher_2024}, it is shown (even in the more general framework of simplicial complexes) that $I_\Gamma = \ann W_\Gamma$, that is, the annihilator scheme structure of the resonance is reduced. As already mentioned in \cite[Example 5.3]{aprodu_higher_2024}, this is not always the case for $\cR^{\operatorname{Fitt}}_\Gamma$. It is thus natural to ask when is the Fitting ideal $\fitt W_\Gamma$ radical. We answer this question in the following subsection.

\subsection{Fitting scheme structure of a monomial resonance}

Let $K \subseteq \bigwedge^2V$ be a monomial subspace and $\Gamma = (\mathsf{V}, \mathsf{E})$, with $\mathsf{V} = \{1,\ldots,n\}$, the simple and undirected graph determined by the pair $(V,K)$. Let $W_\Gamma$ be the associated Koszul module and $S = \bk[x_1,\ldots,x_n]$ be the polynomial ring corresponding to the monomial basis of $V$. Consider also $\overline{\Gamma} = (\mathsf{V},\overline{\mathsf{E}})$ the complement graph of $\Gamma$. Starting from the presentation
\begin{align*}
W_\Gamma = W(V,K) = \operatorname{coker} \left( \displaystyle \textstyle{\bigwedge^3 V} \otimes S(-1) \xrightarrow{\ \pi\circ\delta_3\ } \displaystyle\frac{\textstyle{\bigwedge^2 V}}{K} \otimes S\right)
\end{align*}
we easily infer that the matrix $\Theta_\Gamma$ of the graded $S$-module $W_\Gamma$ is obtained from the third Koszul matrix $\delta_3$ by removing the lines corresponding to the edges of $\Gamma$. Put it differently, $\Theta_\Gamma$ is a $| \overline{\mathsf{E}} | \times \binom{n}{3}$ matrix with entries $0$ or $\pm x_i$, with $i = \overline{1,n}$, described as follows. If for any missing edge $(i,j)\in \overline{\mathsf{E}}$ and any ordered triplet $(k,l,m)$, we denote by $\theta_{ij, klm}$ the element on row $(i,j)$ and column $(k,l,m)$ of the matrix $\Theta_\Gamma$, then 
\begin{equation} \label{eq:pres_matrix}
\theta_{ij, klm} = \left\{
\begin{array}{ll}
      x_m \ \ \ \text{if} \ (i,j) = (k,l) \\
      -x_l \ \ \text{if} \ (i,j) = (k,m)\\
      x_k \ \ \ \ \text{if} \ (i,j) = (l,m)\\
      0 \ \ \ \ \ \ \text{otherwise}
\end{array} 
\right. 
\end{equation}
Note that $\Theta_\Gamma$ has precisely $n-2$ variables on each row and at most $3$ variables on each column, cf. \cite[Section 8]{aprodu_reduced_2024}. 

\medskip

Now, the $n$--dimensional torus $\mathbb{T} = (\bk^*)^n$ acts on $S^{| \overline{\mathsf{E}}|}$ by:
\begin{align*}
(\lambda_1,\ldots,\lambda_n)\cdot(\ldots,f_{(ij)}(x_1,\ldots,x_n),\ldots) = (\ldots, \lambda_i\lambda_jf_{(ij)}(\lambda_1x_1,\ldots,\lambda_nx_n),\ldots)
\end{align*}
If $c_{ijk}$ is the $(i,j,k)$--column of $\Theta_\Gamma$ then
\begin{align*}
(\lambda_1,\ldots,\lambda_n)\cdot c_{ijk} = \lambda_i\lambda_j\lambda_k c_{ijk}
\end{align*}
Thus, we see that the ideal generated by any minor of $\Theta_\Gamma$ is $\mathbb{T}$--invariant under the usual action of $\mathbb{T}$ on the polynomial ring $S$:
\begin{align*}
(\lambda_1,\ldots,\lambda_n)\cdot f(x_1,\ldots,x_n) = f(\lambda_1x_1,\ldots,\lambda_nx_n)
\end{align*}
Therefore, by \cite[Proposition 2.1]{miller_2005}, every minor of $\Theta_\Gamma$ is a monomial. In particular,

\begin{rem} \label{rem_fitt_monomial}
$\operatorname{Fitt}_0(W_\Gamma)$ is a monomial ideal. Moreover, if $\Gamma$ is connected, the Fitting ideal $\operatorname{Fitt}_0(W_\Gamma)$ is generated by monomials of degree $| \overline{\mathsf{E}} |$.
\end{rem}

Since the radical of $\fitt W_\Gamma$ is $I_\Gamma$, our strategy is to analyze the degree of the monomials generating the ideal $I_\Gamma$. To be more precise, we will compare the lowest degree $d_\Gamma$ of a monomial in $I_\Gamma$ with $| \overline{\mathsf{E}} |$. But first, let us observe a few simple facts:

\begin{rem} \label{obs_kn}
If $\Gamma$ is the complete graph $K_n$, then $\fitt W_\Gamma = I_\Gamma = S$.
\end{rem}

\begin{rem} \label{obs_kn_fara_una}
If $\Gamma$ is obtained from $K_n$ by removing an edge $(i,j)$, then $\fitt W_\Gamma = I_\Gamma$ is the ideal generated by all the variables $x_k$, with  $k\neq i,j$.
\end{rem}

\begin{rem} \label{obs_c4}
If $\Gamma$ is the cycle graph on $4$ vertices $C_4$, then the matrix of $W_\Gamma$ is
\[
\Theta_\Gamma = \kbordermatrix{
    & 123 & 124 & 134 & 234 \\
    13 & -x_2 & 0 & x_4 & 0 \\
    24 & 0 & x_1 & 0 & -x_3
  }
\]
and $\fitt W_\Gamma = I_\Gamma = (x_1x_2,x_1x_4,x_2x_3,x_3x_4) = (x_1,x_3) \cap (x_2,x_4)$. 
\end{rem}

In the sequel, we show that if $\Gamma$ is connected, but different from the graphs described in the previous three remarks, then $\fitt W_\Gamma$ is not radical. For $n = 3$, there is nothing to prove. For $n = 4$, the remaining cases are treated in the examples below.

\begin{exmp} \label{ex:n = 4_tree1}
If $\mathsf{E} = \{(1,2), (2,3), (3,4)\}$, that is, $\Gamma$ is the path graph on $4$ vertices, then 
\[
\fitt W_\Gamma = (x_1x_2x_3, x_2x_3x_4, x_2x_3^2, x_2^2x_3) = (x_2) \cap (x_3) \cap (x_1, x_2^2, x_3^2, x_4),
\]
whereas $I_\Gamma = (x_2 x_3) = (x_2) \cap (x_3)$.
\end{exmp}

\begin{exmp} \label{ex:n = 4_tree2}
If $\mathsf{E} = \{(1,2), (1,3), (1,4)\}$, then
\[
\fitt W_\Gamma = (x_1^3, x_1^2x_2, x_1^2x_3, x_1^2x_4) = (x_1^2) \cap (x_1^3, x_2, x_3, x_4),
\]
whereas $I_\Gamma = (x_1)$.    
\end{exmp}

\begin{exmp}
If $\mathsf{E} = \{(1,2), (1,3), (2,3), (3,4)\}$, then
\[
\fitt W_\Gamma = (x_1x_3, x_2x_3, x_3^2) = (x_1, x_2, x_3^2) \cap (x_3),
\]
whereas $I_\Gamma = (x_3)$.
\end{exmp}

Recall that the degree of a vertex $i$ is the number of edges that are incident to $i$. A vertex of degree 0 is called \emph{isolated} and a vertex of degree $1$ is called \emph{terminal}.

\begin{lem} \label{lem_obs}
\hfill
\begin{itemize}
    \item [(a)] If $\overline{\Gamma}$ has an isolated vertex, then $d_\Gamma = 1$.
    \item [(b)] If $\overline{\Gamma}$ is disconnected, then $d_\Gamma \le 2$.
    \item [(c)] If $\overline{\Gamma}$ has a terminal vertex, then $d_\Gamma \le 2$.
    \item [(d)] If $\overline{\Gamma}$ has a vertex of degree $2$, then $d_\Gamma \le 4$.
\end{itemize}
\end{lem}

\begin{proof}
(a) Simply note that the isolated vertex $i$ belongs to any minimal disconnecting set of $\Gamma$ and therefore, $x_i \in I_\Gamma$.\\
(b) Since its complement is disconnected, $\Gamma$ is the join $\Gamma_1 * \Gamma_2$ of two subgraphs $\Gamma_1 = (\mathsf{V_1}, \mathsf{E_1})$ and $\Gamma_2 = (\mathsf{V_2}, \mathsf{E_2})$, that is $\mathsf{V} = \mathsf{V}_1 \sqcup \mathsf{V}_2$ and $\mathsf{E} = \mathsf{E}_1 \cup \mathsf{E}_2 \cup \{(i,j) : i \in \mathsf{V}_1, j \in \mathsf{V}_2\}$. Note that any subgraph of $\Gamma$ containing a vertex $i$ of $\Gamma_1$ and a vertex $j$ of $\Gamma_2$ is connected. Thus, by Theorem \ref{thm_structure_ideal_res}, all the monomials $x_i x_j$, with $i \in \mathsf{V}_1$, $j \in \mathsf{V}_2$ belong to $I_\Gamma$.\\
(c) If $\overline{\Gamma}$ has a terminal vertex $i$, then there is only one minimal disconnecting subset of $\Gamma$ which does not contain $i$. Therefore, using Theorem \ref{thm_structure_ideal_res}, the ideal $I_\Gamma$ contains a monomial of the form $x_i x_j$, for some $j\in \mathsf{V}$.\\
(d) If $\overline{\Gamma}$ has a vertex $i$ connected with only two other vertices, then there are at most $3$ minimal disconnecting subsets of $\Gamma$ which do not contain $i$. By Theorem \ref{thm_structure_ideal_res}, the ideal $I_\Gamma$ contains a monomial of degree $4$, that is $d_\Gamma \le 4$.
\end{proof}

\begin{prop} \label{prop_neredus}
If $n \ge 5$,  $\Gamma$ is connected and $| \overline{\mathsf{E}} | \ge 2$, then $d_\Gamma < | \overline{\mathsf{E}} |$. In particular, the Fitting ideal $\fitt W_\Gamma$ is not radical.
\end{prop}

\begin{proof}
To begin with, note that the statement is obvious if $| \overline{\mathsf{E}} | > n$. Also, if $| \overline{\mathsf{E}} | = 2$, then $\overline{\Gamma}$ has an isolated vertex, so $d_\Gamma = 1$, by Lemma \ref{lem_obs}. Hence, we may assume that $| \overline{\mathsf{E}} | \in \{3, \ldots, n\}$.

If $| \overline{\mathsf{E}} | < n - 1$, then $\overline{\Gamma}$ is disconnected and hence $d_\Gamma \le 2$, by Lemma \ref{lem_obs}.

If $\overline{\Gamma}$ is connected and has $n - 1$ edges or if $\overline{\Gamma}$ is connected, has $n$ edges and it is not the cycle $C_n$, then $\overline{\Gamma}$ has a terminal vertex. Thus, by Lemma \ref{lem_obs}, $d_\Gamma \le 2$.

For the last case, namely $\overline{\Gamma} = C_n$, note that any vertex of $\overline{\Gamma}$ is connected with precisely two other vertices. Hence, by Lemma \ref{lem_obs}, $d_\Gamma \le 4$. Since in this case $| \overline{\mathsf{E}} | = n \ge 5$, we are done.
\end{proof}

Summarizing, we obtain a precise description of the graphs for that the Fitting scheme structure of the associated resonance is reduced.

\begin{thm} \label{thm_fitting}
The scheme $\cR^{\operatorname{Fitt}}_\Gamma$ is reduced if and only if one of the conditions below is satisfied:
\begin{itemize}
    \item [(i)] $\Gamma$ is disconnected;
    \item [(ii)] $\Gamma$ is the complete graph with $n$ vertices $K_n$;
    \item [(iii)] $\Gamma$ is obtained from $K_n$ by removing an edge;
    \item [(iv)] $\Gamma$ is the cycle graph with $4$ vertices $C_4$.
\end{itemize}
\end{thm}

In contrast to the annihilator ideal of $W_\Gamma$, a complete description of the primary decomposition of $\fitt W_\Gamma$ in terms of the combinatorial structure of $\Gamma$ is not yet known. A tractable starting point is the case of trees. Therefore, in what follows, we will consider a tree $T = (\mathsf{V}, \mathsf{E})$ on $n \ge 3$ vertices, that is, a connected graph with exactly $n - 1$ edges. Note that removing a vertex $i \in \mathsf{V}$ from $T$ (and implicitly, all the edges incident to $i$) yields a graph with precisely $\deg_T(i)$ connected components, where $\deg_T(i)$ denotes the degree of the vertex $i$ in $T$. We associate the following monomial to $T$:
\begin{equation} \label{eq:m_T}
m_T = \prod_{i=1}^n x_i^{\deg_T(i) - 1} \in S .
\end{equation}

\begin{notation}
For brevity, we denote by $t_{m} : = \binom{m+1}{2}$ the $m$-th \emph{triangular} number, for any integer $m \ge 0$. With this notation, the complement of $T$ has precisely $t_{n-1} - (n-1) = t_{n-2}$ edges.
\end{notation}
    
\begin{obs}
As the Fitting ideal of a finitely generated $S$-module is invariant under any elementary transformation of its presentation matrix, we shall, by a slight abuse of notation, denote by $\Theta_\Gamma$ any matrix obtained from the presentation matrix (\ref{eq:pres_matrix}) of $W_\Gamma$ by permuting its rows or columns, or by multiplying them by $-1$.
\end{obs}

\begin{lem} \label{lem_tree_part1}
Any monomial of the ideal $\fitt W_T$ is divisible by $m_T$.
\end{lem}

\begin{proof}
It is enough to show that any monomial of $\fitt W_T$ is a multiple of $x_1^{\deg_T(1) - 1}$.
We start by writing the matrix $\Theta_T$ as
\[
\Theta_T =
\kbordermatrix{
    & (1,i,j) & & (1,i,'j') & & (i_1,j_1,k_1) \\[2pt]
  (i,j) & \mathbf{I} \cdot x_1 & \vrule & 0 & \vrule & \Theta_{T \setminus \{1\}} \\[2pt]
  \cline{2-6}
  (1,k) & \star & \vrule & \star & \vrule & 0
}
\]
where
\begin{itemize}
    \item [--] $T \setminus \{1\}$ denotes the full subgraph of $T$ on the vertex set $\mathsf{V} \setminus \{1\}$,
    \item [--] $\mathbf{I}$ denotes the identity matrix,
    \item [--] the rows $(i,j)$ and the columns $(1,i,j)$ correspond to the missing edges of $T$ that are not incident to $1$, that is the missing edges of $T \setminus \{1\}$;
    \item [--] the rows $(1,k)$ correspond to the missing edges of $T$ that are incident to $1$,
    \item [--] the columns $(1,i',j')$ correspond to the edges $(i',j')$ of $T \setminus \{1\}$,
    \item [--] the columns $(i_1,j_1,k_1)$ correspond to all the ordered triplets not containing $1$.
\end{itemize}
As already mentioned, $\Theta_T$ has precisely $r := t_{n-2}$ rows and the set $I$ formed by the rows of type $(i,j)$ has cardinality $r': = t_{n-2} - n + 1 + \deg_T(1)$. Now let $\Theta'$ be a maximal square submatrix of $\Theta_T$, formed by $r$ columns, denoted by $c_1, \ldots, c_r$. By the generalized Laplace expansion,
\[
\det \Theta' = \sum_{J} \pm\  \det \Theta'[I, J] \cdot \det \Theta'[I^c, J^c]
\]
where the sum is taken over all subsets $J \subseteq \{c_1, \ldots, c_r\}$ of size $|J| = r'$. The square submatrix $\Theta'[I, J]$ is obtain from $\Theta'$ by selecting the rows indexed by $I$ and columns indexed by $J$, while the complementary submatrix $\Theta'[I^c, J^c]$ is obtained by deleting those rows and columns from $\Theta'$. The $\pm$ sign depends on the choice of $J$, but it is not relevant for our purposes.

Let us now analyze a matrix of type $\Theta'[I,J]$. If $J_0 \subseteq J$ is formed by the columns of type $(1,i,j)$ and $I_0 \subseteq I$ denotes the set of rows $(i,j)$ such that $(1,i,j)$ belongs to $J_0$, then
\[
\det \Theta'[I,J] = x_1^{| I_0 |} \cdot \det \Theta'[I_0^c, J_0^c],
\]
where the matrix $\Theta'[I_0^c, J_0^c]$ is obtained from $\Theta'[I,J]$ by deleting the rows $I_0$ and the columns $J_0$. Remark that $\det \Theta'[I_0^c, J_0^c] \in \fitt W_{\Gamma'}$, where $\Gamma'$ is the graph obtained from $T \setminus \{1\}$ by adding the edges corresponding to the rows of $I_0$. Since $\fitt W_{\Gamma'} \neq (0)$ if and only if $\Gamma'$ is connected, we infer that, for the minor $\det \Theta'[I,J]$ of $\Theta_T$ to be non-zero, $| I_0 |$ has to be at least $\deg_T(1) - 1$ and hence $x_1^{\deg_T(1) - 1}$ divides any monomial in the ideal $\fitt W_T$.
\end{proof}

\begin{lem} \label{lem_tree_part2}
The monomial $x_i^{t_{n-3}} \cdot m_T$ belongs to $\fitt W_T$, for any $i \in \mathsf{V}$.
\end{lem}

\begin{proof}
We will proceed by induction on $n$. The case $n = 3$ is trivial and the case $n = 4$ has been verified in Examples \ref{ex:n = 4_tree1} and \ref{ex:n = 4_tree2}. Let us suppose that $n \ge 5$ and the statement is true for $n - 1$. Without loss of generality, we may further assume that $T$ has a particular labeling, as follows. First, we take $1$ as a terminal vertex of $T$, with $(1,2) \in \mathsf{E}$ as the unique edge incident to $1$. Then, for each $j \in \{3, \ldots, n\}$, when listing the edges of $T$ in lexicographic order, the $(j - 1)$-th edge is $(i,j)$, for some $i < j$. The matrix $\Theta_T$ has the following form
\[
\Theta_T =
\kbordermatrix{
    & (1,i,j) & & (1,i,'j') & & (i_1,j_1,k_1) \\[2pt]
  (1,j) & U & \vrule & \star & \vrule & 0 \\[2pt]
  \cline{2-6}
  (i',j') & 0 & \vrule & \mathbf{I} \cdot x_1 & \vrule & \Theta_{T \setminus \{1\}}
}
\]
where
\begin{itemize}
    \item [--] $U$ is an upper triangular matrix of size $n - 2$, 
    \item [--] the rows $(1,j)$ correspond to the missing edges $(1,3), \ldots, (1,n)$ of $T$,
    \item [--] the columns $(1,i,j)$ correspond to the edges $(i,j)$ of $T$ that are not incident to $1$,
    \item [--] the rows $(i',j')$ and the columns $(1,i',j')$ correspond to the missing edges $(i',j')$ of $T$ that are not incident to $1$,
    \item [--] the columns $(i_1,j_1,k_1)$ correspond to the all ordered triplets not containing $1$.
\end{itemize}
Note that, for any $j\in \{3, \ldots, n\}$, the entry on the main diagonal of $U$ corresponding to the intersection of the row $(1,j)$ and the column $(1,i,j)$ is $-x_i$. Thus, the minor corresponding to the first $t_{n-2}$ columns of $\Theta_T$ is $ \pm x_1^{t_{n-3}} \cdot m_T$.

It remains to show that $x_i^{t_{n-3}} \cdot m_T \in \fitt W_T$, for $i \ge 2$. Since $1$ is a terminal vertex of $T$, we may apply the induction hypothesis to the tree $T': = T \setminus \{1\}$ to produce, for each $i \ge 2$, a maximal square submatrix $\Theta'_i$ of the $\Theta_{T'}$ such that
\[
\det \Theta'_i = \alpha_i \cdot x_i^{t_{n-4}} \cdot m_{T'},
\]
where $\alpha_i$ is some non-zero constant. Note that $m_T = x_2 \cdot m_{T'}$. Considering the shape of $\Theta_T$, it suffices to form, for each $i \ge 2$, a square submatrix $\Theta_i$ of size $n - 2$ with the first $n - 2$ rows and $n - 2$ of the first $t_{n-2}$ columns of $\Theta_T$, such that
\[
\det \Theta_i = \pm x_2 x_i^{n-3}.
\]
For $i = 2$, we form $\Theta_2$ by choosing the columns $(1,2,j)$, with $j\in \{3, \ldots, n\}$. Each row and each column of $\Theta_2$ has precisely one non-zero entry: $-x_2$. For  $i \ge 3$, we form $\Theta_i$ by choosing the columns $(1,2,i)$ and $(1,i,j)$, with $j\neq 1,2,i$. After expanding with respect to the column $(1,2,i)$, we will get a matrix such that the only non-zero entry on the column $(1,i,j)$ is $\pm x_i$. The conclusion follows.
\end{proof}

As a consequence of Lemmas \ref{lem_tree_part1} and \ref{lem_tree_part2}, we obtain the following result concerning the primary decomposition of $\fitt W_T$.

\begin{thm} \label{thm_primary_decomp_trees}
Let $T$ be a tree on $n \ge 4$ vertices and $\cD_T$ the set of vertices of $T$ of degree at least $2$. The ideal $\fitt W_T$ has the following primary decomposition
\[
\fitt W_T = \bigcap_{i \in \cD_T} (x_i)^{\deg_T(i)-1} \cap J_T
\]
where $J_T$ is a non-radical monomial ideal, whose radical equals $\mathfrak{m}:=(x_1, \ldots, x_n)$.
\end{thm}

% \begin{cor}
% The path graph have reduced projectivized resonance, when endowed with the Fitting structure.
% \end{cor}

Guided by the case of trees, we propose the following description of the minimal components of $\fitt W_\Gamma$, for an arbitrary connected graph $\Gamma$, cf. Theorem \ref{thm_structure_ideal_res}.

\begin{conj} \label{conj_minimal components}
The minimal components of $\fitt W_\Gamma$ are $I_\mathsf{U}^{d_\mathsf{U} - 1}$, where $\mathsf{U} \in \cD_\Gamma$ and $d_{\mathsf{U}}$ is the number of connected components of $\Gamma \setminus \mathsf{U}$.
\end{conj}

We verified this conjecture computationally using \emph{Macaulay2} \cite{M2} for small values of $n$. The computations also suggest that, in general, there may exist several embedded components, rather than a single one. Determining the precise nature of these embedded components in terms of the combinatorics of the graph remains an open problem. An example exhibiting $3$ embedded components is given below.

\begin{exmp} \label{ex:3embedded}
If $\mathsf{V} = \{1, \ldots, 5\}$ and $\mathsf{E} = \{(1,2), (1,3), (1,4), (1,5), (2,3), (3,4), (4,5) \}$, then

\noindent
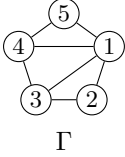
\begin{tabular}{@{}c@{\hspace{1cm}}c@{}}
  \raisebox{-0.5\height}{%
    \begin{tikzpicture}[scale = 0.7,
      every node/.style = {circle, draw, inner sep = 1pt, minimum size = 4mm, font = \small}]
      \foreach \i in {1,...,5} {
        \node (\i) at (90 - 72*\i:0.9cm) {\i};
      }
      \draw (1)--(2);
      \draw (1)--(3);
      \draw (1)--(4);
      \draw (1)--(5);
      \draw (2)--(3);
      \draw (3)--(4);
      \draw (4)--(5);
      \node[draw = none, below = 6pt, font = \normalsize] at (270:0.9cm) {$\Gamma$};
    \end{tikzpicture}%
  }
  &
  $\vcenter{
    \hbox{
      \begin{minipage}{0.7\textwidth}
        \[
        \begin{aligned}
          \fitt W_\Gamma = {} &(x_1, x_3) \cap (x_1, x_4) \cap (x_1^2, x_2, x_3, x_4) \cap (x_1^2, x_3, x_4, x_5) \\
                             &\cap (x_1^3, x_1^2 x_3, x_1^2 x_4, x_1 x_3 x_4, x_2, x_3^2, x_4^2, x_5)
        \end{aligned}
        \]
      \end{minipage}
    }
  }$
\end{tabular}
\end{exmp}

Recall that a \emph{spanning tree} of the graph $\Gamma$ is a subgraph of $\Gamma$ on the full vertex set $\mathsf{V}$ that is a tree. Let us denote by $\cT(\Gamma)$ the set of all spanning trees of $\Gamma$. Define
\[
I_{\cT(\Gamma)} = \left(m_T : T \in \cT(\Gamma)\right)
\]
the ideal generated by the monomials $m_T$, see (\ref{eq:m_T}), associated to all spanning trees $T$ of the graph $\Gamma$.

Provided that $| \overline{\mathsf{E}} | \ge n - 2$, one can strengthen Conjecture \ref{conj_minimal components} to

\begin{conj} \label{conj_fitt_ideal}
If $\Gamma = (\mathsf{V}, \mathsf{E})$ is a connected graph such that $| \overline{\mathsf{E}} | \ge n - 2$, then
\[
\fitt W_\Gamma = I_{\cT(\Gamma)} \cdot \mathfrak{m}^{| \overline{\mathsf{E}} | - n + 2}
\]
where $\mathfrak{m} = (x_1, \ldots, x_n)$.
\end{conj}

\begin{prop} \label{strong_weak_conj}
If $\Gamma = (\mathsf{V}, \mathsf{E})$ is a connected graph such that $| \overline{\mathsf{E}} | \ge n - 2$, then Conjecture \ref{conj_fitt_ideal} implies Conjecture \ref{conj_minimal components}.
\end{prop}

\begin{proof}
Let $\mathsf{U} = \{i_1, \ldots, i_p\}$ be a minimal disconnecting set of $\Gamma$ and $d_{\mathsf{U}} \ge 2$ the number of connected components of $\Gamma \setminus \mathsf{U}$. The following two assertions are immediate.
\begin{itemize}
    \item [(i)] For any $T \in \cT(\Gamma)$ one has $\sum_{k = 1}^p \deg_T(i_k) \ge p + d_{\mathsf{U}} - 1$.
    \item [(ii)] For any positive integers $a_1, \ldots, a_p$ such that $\sum_{k = 1}^p a_k = p + d_{\mathsf{U}} - 1$, there exists a tree $T_0 \in \cT(\Gamma)$ such that $\deg_{T_0}(i_k) = a_k$, for all $k \in \{1, \ldots, p\}$.
\end{itemize}
In the light of these facts, we infer that any generator of $I_{\cT(\Gamma)}$ is a multiple of a degree $d_{\mathsf{U}} - 1$ monomial in the variables $x_{i_1}, \ldots, x_{i_p}$ and any monomial of degree $d_{\mathsf{U}} - 1$ in the variables $x_{i_1}, \ldots, x_{i_p}$ divides one of the generators of $I_{\cT(\Gamma)}$ and, moreover, their quotient is supported on the set of variables $\{x_h : h \in \mathsf{V} \setminus \mathsf{U}\}$. Thus, if Conjecture \ref{conj_fitt_ideal} holds, the minimal component of $\fitt W_\Gamma$ corresponding to the minimal prime $I_{\mathsf{U}} = (x_{i_1}, \ldots, x_{i_p})$ is the one described in Conjecture \ref{conj_minimal components}, namely $I_{\mathsf{U}}^{d_{\mathsf{U}} - 1}$.
\end{proof}

A \emph{Macaulay2} script that generates an arbitrary connected graph with $n$ vertices and $m$ edges, and verifies Conjecture \ref{conj_fitt_ideal} is provided below. Due to computational constraints, the script is practical only if $n \le 6$.

\begin{lstlisting}[caption={A \emph{Macaulay2} script for Conjecture~\ref{conj_fitt_ideal}}]
loadPackage "EdgeIdeals"
loadPackage "ThinSincereQuivers"

n = 5; -- #vertices
m = 6; -- #edges
S = QQ[x_1..x_n];
G = randomGraph(S,m);
while isConnected G == 0 do{
    G = randomGraph(S, m);
};
A = adjacencyMatrix G;
c = 0;
L = {};
Q0 = {};

for j from 1 to n-1 do{
    for i from 0 to j-1 do{
        if A_(i,j) == 1 then{
            L = append(L,c);
            Q0 = append(Q0,{i,j});
        };
        c = c + 1
    }
}
Q0 = sort Q0;
Q = toricQuiver(Q0);
Arb = allSpanningTrees(Q);

C = koszul(matrix{{x_1..x_n}});
delta3 = C.dd_3;
M = submatrix'(delta3,L,);
W = coker M
Fitt = fittingIdeal(0,W)

p = 1;
for i from 1 to n do{
    p = p*x_i;
}
maxId = ideal(x_1..x_n);
ArbId = ();

for i from 0 to length(Arb)-1 do{
    T = Arb_i;
    mon = 1;
    for k from 0 to n-2 do{
        E = Q0_T_k;
        mon = mon*x_(E_0+1)*x_(E_1+1);
    };
    mon = mon//p;
    ArbId = ideal(ArbId, mon);
}
if numRows M >= n-2 then 
    Fitt == ArbId*maxId^(numRows M - n + 2)
\end{lstlisting}

\subsection{Non-monomial subspaces with monomial resonance}

In the sequel, $\Gamma = (\mathsf{V}, \mathsf{E})$ will be a connected simple graph with $\mathsf{V} = \{1, \ldots, n\}$ and $n \ge 4$. We will denote the resonance variety and the Koszul module associated with $\Gamma$ by $\cR_\Gamma$ and $W_\Gamma$, respectively. Let us also consider $\{e_1, \ldots, e_n \}$ the basis of $V$ corresponding to $\Gamma$ and $\{f_1, \ldots, f_n \}$ its dual basis in $V^\vee$.

It is natural to ask whether the resonance can distinguish between monomial and non-monomial subspaces, or, put it differently, if one can recover the graph $\Gamma$ starting from the resonance. It is the aim of this subsection to partially answer this question. We first outline some methods for constructing non-monomial subspaces that have a monomial resonance.

\begin{prop} \label{prop_triangle}
If the graph $\Gamma$ contains a triangle (i.e. a cycle on $3$ vertices), there exists a non-monomial subspace $K$ such that $\cR_\Gamma = \cR(V,K)$.
\end{prop}

\begin{proof}
Without loss of generality, we assume that the triangle of $\Gamma$ is formed by the vertices $\{1,2,3\}$, that is $(1,2), (1,3), (2,3) \in \mathsf{E}$. Since $\Gamma$ is connected, we may also assume that $(3,4) \in \mathsf{E}$. 

Let us define $K$ to be the subspace of $\bigwedge^2V$ generated by $e_1 \wedge e_2 + e_3 \wedge e_4$ and all the totally decomposable vectors $e_i \wedge e_j$ such that $(i,j) \in \mathsf{E} \setminus \{(1,2), (3,4) \}$. Note that $K^\perp$ is generated by $f_1 \wedge f_2 - f_3 \wedge f_4$ and all the totally decomposable vectors $f_i \wedge f_j$ such that $(i,j) \notin \mathsf{E}$. It is clear that $K$ is not a monomial subspace. In what follows, we prove that the resonance of $K$ coincides (as a variety) with the resonance of $\Gamma$. 

By definition, a non-zero vector $a  = \sum_{i=1}^n a_i f_i \in V^\vee$ belongs to $\cR_\Gamma$ if and only if there exists a non-zero vector $b = \sum_{i=1}^n b_i f_i$, non-proportional to $a$, such that
\begin{equation} \label{eq:graph}
    a_i b_j = a_j b_i, \ \forall \ (i,j) \in \mathsf{E}
\end{equation}
Now, a non-zero vector $a  = \sum_{i=1}^n a_i f_i \in V^\vee$ belongs to $\cR(V,K)$ if and only if there exists a non-zero vector $b = \sum_{i=1}^n b_i f_i$, non-proportional to $a$, such that
\begin{equation} \label{eq:non-monomial}
    \begin{cases}
        a_1 b_2 - a_2 b_1 = a_3 b_4 - a_4 b_3 \\
        a_i b_j = a_j b_i, \ \forall \ (i,j) \in \mathsf{E}\setminus \{(1,2), (3,4) \}
    \end{cases}
\end{equation}
It is clear that (\ref{eq:graph}) $\implies$ (\ref{eq:non-monomial}) and hence $\cR_\Gamma \subseteq \cR(V,K)$. Conversely, let $a,b$ be two non-proportional vectors satisfying the equations (\ref{eq:non-monomial}) and suppose that $a_1 b_2 - a_2 b_1 \neq 0$. Then, since $a_3 b_4 \neq a_4 b_3$, $a_1 b_3 = a_3 b_1$ and $a_2 b_3 = a_3 b_2$, we immediately obtain that $b_1, b_2$ and $b_3$ should be non-zero. But then, the equations $a_1 b_3 = a_3 b_1$ and $a_2 b_3 = a_3 b_2$ ensure that $a_1 b_2 = a_2 b_1$, leading to a contradiction. Consequently, $\cR(V,K) \subseteq \cR_\Gamma$. 
\end{proof}

We emphasize in the example below that the equality obtained in Proposition \ref{prop_triangle} is only set-theoretical. When endowed with the natural schemes structures described in Section \ref{sec:2}, the resonance scheme $\cR^{\ann}_\Gamma$ (or, respectively, $\cR^{\operatorname{Fitt}}_\Gamma$) may be different from the scheme $\cR^{\ann}(V,K)$ (or $\cR^{\operatorname{Fitt}}(V,K)$).

\begin{exmp} \label{ex:non-reduced}
Let us take $n = 4$ and $\Gamma = (\mathsf{V}, \mathsf{E})$, where $\mathsf{E} = \{(1,2), (1,3), (2,3), (3,4) \}$. If the subspace $K$ is generated by $e_1 \wedge e_2 + e_3 \wedge e_4, e_1 \wedge e_3$ and $e_2 \wedge e_3$, then, according to Proposition \ref{prop_triangle}, the resonance variety $\cR(V,K)$ will coincide with the resonance of $\Gamma$, which is $\operatorname{Span}\{f_1, f_2, f_4\}$. However, $\ann W(V,K) = (x_3^2) \neq (x_3) = \ann W_\Gamma$ and thus, although the resonance scheme $\cR^{\ann}_\Gamma$ is reduced, $\cR^{\ann}(V,K)$ is not. As for the Fitting scheme structures, we remark that $\fitt W(V,K) = (x_3^2) \cap (x_1, x_2, x_3^3, x_4)$, whereas $\fitt W_\Gamma = (x_3) \cap (x_1, x_2, x_3^2)$.
\end{exmp}

Another method of constructing non-monomial subspaces having a monomial resonance is presented in the next result.

\begin{prop} \label{prop_terminal_vertex}
If the graph $\Gamma$ has a terminal vertex, there exists a non-monomial subspace $K$ such that $\cR_\Gamma = \cR(V,K)$.
\end{prop}

\begin{proof}
Without loss of generality, suppose that $1$ is a terminal vertex and $(1,2) \in \mathsf{E}$. If $\Gamma$ contains a triangle, the conclusion already follows from Proposition \ref{prop_triangle}, so we may also assume that, for instance, $(3,4) \notin \mathsf{E}$.

Define $K$ to be the subspace of $\bigwedge^2V$ generated by $e_1 \wedge e_2 + e_3 \wedge e_4$ and all the totally decomposable vectors $e_i \wedge e_j$ such that $(i,j) \in \mathsf{E} \setminus \{(1,2)\}$. Equivalently, $K^\perp$ is generated by $f_1 \wedge f_2 - f_3 \wedge f_4$ and all the totally decomposable vectors $f_i \wedge f_j$ such that $(i,j) \notin \mathsf{E}$.

Let $a  = \sum_{i=1}^n a_i f_i \in V^\vee$ be a non-zero vector. Note that $a \in \cR_\Gamma$ if and only if there exists $b = \sum_{i=1}^n b_i f_i$, which is non-proportional to $a$, such that
\begin{equation} \label{eq:graph_terminal}
    a_i b_j = a_j b_i, \ \forall \ (i,j) \in \mathsf{E}
\end{equation}
and $a \in \cR(V,K)$ if and only if there exists a vector $b' = \sum_{i=1}^n b'_i f_i$, which is non-proportional to $a$, such that
\begin{equation} \label{eq:non-monomial_terminal}
    \begin{cases}
        a_1 b'_2 - a_2 b'_1 = a_3 b'_4 - a_4 b'_3 \\
        a_i b'_j = a_j b'_i, \ \forall \ (i,j) \in \mathsf{E}\setminus \{(1,2)\}
    \end{cases}
\end{equation}
If $a_2 = 0$, then $a \in \cR_\Gamma \cap \cR(V,K)$, since we may take $b = b' = f_1$. Therefore, we assume that $a_2 \neq 0$.

Suppose that $a \in \cR_\Gamma$. Then, there exists a vector $b \in V^\vee$, such that $a \wedge b \neq 0$ and the components of $b$ with respect to the basis $\{f_1, \ldots, f_n\}$ satisfy the relations (\ref{eq:graph_terminal}). Let us take $b' = \sum_{i=1}^n b'_i f_i$, where $b'_i = b_i$, for all $i \ge 2$ and
\[
b'_1 = \frac{1}{a_2}(a_1 b_2 - a_3 b_4 + a_4 b_3)
\]
It is then straightforward to check that $a \wedge b' \neq 0$ and the components of $b'$ satisfy the relations (\ref{eq:non-monomial_terminal}). Therefore, $a \in \cR(V,K)$.

Suppose now that $a \in \cR(V,K)$. Then, there exists a vector $b' \in V^\vee$, such that $a \wedge b' \neq 0$ and the components of $b'$ with respect to the basis $\{f_1, \ldots, f_n\}$ satisfy the relations (\ref{eq:non-monomial_terminal}). Consider now $b = \sum_{i=1}^n b_i f_i$, where $b_i = b'_i$, for all $i \ge 2$ and
\[
b_1 = \frac{a_1 b_2}{a_2}
\]
Note that $a \wedge b \neq 0$ and the components of $b$ satisfy the relations (\ref{eq:graph_terminal}), hence $a \in \cR_\Gamma$.
\end{proof}

\begin{exmp}
By taking $n = 4$ and $\Gamma = (\mathsf{V}, \mathsf{E})$, with $\mathsf{E} = \{(1,3), (2,3), (3,4) \}$ and by choosing $4$ as a terminal vertex, the recipe described in Proposition \ref{prop_terminal_vertex} yields the same linear subspace $K$ as in Example \ref{ex:non-reduced}. Hence, the annihilator scheme structures of $\cR_\Gamma$ and $\cR(V,K)$ are different: the former is always reduced, while the latter is non-reduced. The Fitting scheme structures, however, coincide: $\fitt W_\Gamma = (x_3^2) \cap (x_1, x_2, x_3^3, x_4) = \fitt W(V,K)$.
\end{exmp}

\begin{prop} \label{prop_disconnecting_vertex}
If the graph $\Gamma$ has a disconnecting vertex, there exists a non-monomial subspace $K$ such that $\cR_\Gamma = \cR(V,K)$.
\end{prop}

\begin{proof}
To begin with, note that $\Gamma$ is either the path graph $P_n$ or one of the disconnecting vertices of $\Gamma$ has degree at least $3$. If $\Gamma = P_n$, we may invoke Proposition \ref{prop_terminal_vertex} to conclude.

Otherwise, we may assume without loss of generality that $1$ is a disconnecting vertex and $(1,2), (1,3), (1,4) \in \mathsf{E}$, but $(2,3) \notin \mathsf{E}$. Consider $\Gamma' = (\mathsf{V}, \mathsf{E'})$, where $\mathsf{E'} = \mathsf{E} \cup \{(2,3)\}$ and note that $\Gamma$ and $\Gamma'$ have the same disconnecting sets of vertices. Indeed, the only aspect that may require attention is that $1$ remains a disconnecting vertex in $\Gamma'$. This is ensured by the existence of the edge $(1,4) \in \mathsf{E}$.  Therefore, the resonance of $\Gamma$ is the same with the resonance of $\Gamma'$. Since $\Gamma'$ contains a triangle, we are done by Proposition \ref{prop_triangle}. 
\end{proof}

The first graph which satisfy none of the hypotheses of the previous results is the cycle graph $C_4$. This is a case where the resonance can recover the graph.

\begin{prop} \label{prop_C4}
If $\Gamma = C_4$, there exists no non-monomial subspace $K \subseteq \bigwedge^2V$ such that $\mathcal{R}_\Gamma = \mathcal{R}(V,K)$.
\end{prop}

\begin{proof}
Let $K \subseteq \bigwedge^2V$ such that $\cR(V,K) = \operatorname{Span}\{f_1, f_3\} \cup \operatorname{Span}\{f_2, f_4 \}$. We prove that $K^\perp$ is generated by $f_1\wedge f_3$ and $f_2 \wedge f_4$.

First, note that $f_1 \wedge f_3 \in K^\perp$, for otherwise $f_1$ would resonate with a non-trivial linear combination of $f_2$ and $f_4$, that is $f_1 \wedge (\alpha f_2 + \beta f_4) \in K^\perp$. But this condition would imply that $f_1 \wedge (f_1 + \alpha f_2 + \beta f_4) \in K^\perp$ and hence $f_1 + \alpha f_2 + \beta f_4 \in \cR(V,K)$, which is false. By similar arguments, $f_2 \wedge f_4 \in K^\perp$.

By the same type of reasoning, it is also immediate to see that $f_i \wedge f_j \notin K^\perp$ for all $(i,j) \in \mathsf{E}$. It remains to prove that any non-trivial linear combination
\begin{equation} \label{eq:resonant}
    \sum_{(i,j) \in \mathsf{E}} \alpha_{ij} f_i \wedge f_j
\end{equation}
does not belong to $K^\perp$. In this regard, let us consider the linearly independent vectors $a = \alpha_{14} f_1 + f_2 + \alpha_{34} f_3$ and $b = -\alpha_{12} f_1 + \alpha_{23} f_3 + f_4$ and observe that
\[
a \wedge b = \sum_{(i,j) \in \mathsf{E}} \alpha_{ij} f_i \wedge f_j + (\alpha_{12}\alpha_{34} + \alpha_{14}\alpha_{23}) f_1 \wedge f_3 + f_2 \wedge f_4
\]
Since $f_1 \wedge f_3, f_2 \wedge f_4 \in K^\perp$, we deduce that the sum (\ref{eq:resonant}) belongs to $K^\perp$ if and only if $a$ and $b$ belong to the resonance variety $\cR(V,K)$. However, as we have seen, this is possible if and only if $\alpha_{ij} = 0$ for all $(i,j) \in \mathsf{E}$. The conclusion follows.
\end{proof}

\begin{ques}
Is $C_4$ the only graph with the property above?
\end{ques}

\printbibliography

\end{document}